\documentclass[11pt]{amsart}

\usepackage[margin=1in]{geometry}
\usepackage{amsmath,amssymb,amsthm}
\usepackage{enumitem}
\usepackage{url}

\numberwithin{equation}{section}

\newtheorem{theorem}{Theorem}[section]
\newtheorem{proposition}[theorem]{Proposition}
\newtheorem{lemma}[theorem]{Lemma}
\newtheorem{corollary}[theorem]{Corollary}

\theoremstyle{definition}
\newtheorem{definition}[theorem]{Definition}
\newtheorem{example}[theorem]{Example}
\newtheorem{remark}[theorem]{Remark}

\title[Asymptotic Universal Koszulity]{Asymptotic Universal Koszulity in Galois Cohomology}

\author{Marina Palaisti}
\address{Huron University College, 1349 Western Rd, London, Ontario, N6G 1H3, Canada}
\email{mpalaist@huron.uwo.ca}
\date{\today}

\begin{document}
	
	\maketitle

	\begin{abstract}
		We introduce the notion of asymptotic universal Koszulity for graded-commutative algebras generated in degree~$1$, which formalizes the idea that an infinite-dimensional algebra can be approximated by a filtered system of finite-type universally Koszul subalgebras. We establish both a filtered-colimit criterion and an intrinsic local criterion, formulated in terms of finite-dimensional subspaces of the degree-$1$ component. We then apply this framework to continuous mod-$p$ cohomology of profinite groups. Under a natural cohomological colimit hypothesis, together with universal Koszulity of the finite-type quotient pieces, we prove that the full cohomology algebra is asymptotically universally Koszul. For finitely generated pro-$p$ quotients with injective inflation on cohomology, we identify the image of the quotient cohomology algebra with the canonical quadratic subalgebra generated by the image of degree-$1$ classes. Finally, we formulate a conditional local-global criterion that isolates the finite-dimensional detection needed to deduce asymptotic universal Koszulity in arithmetic settings.
	\end{abstract}
	
	\noindent{\textbf{Keywords:}} Galois cohomology; asymptotic universal Koszulity; pro-$p$ groups; quadratic algebras; local-global descent; patching methods\\
	
	\noindent{\textbf{Classification (MSC2020):}} 12G05; 12F10; 16S37; 20E18

	\section{Introduction}
	
	Let \(p\) be a fixed prime and let \(K\supset \mu_p\) be a field. Write
	\[G_K(p) := \mathrm{Gal}(K(p)/K)\quad\text{and}\quad A_K^\bullet := H^\bullet(G_K(p),\mathbb{F}_p) \]
	for the maximal pro-\(p\) Galois group of \(K\) and its mod-\(p\) cohomology ring respectively. In many arithmetic situations, \(A_K^\bullet\) is a quadratic algebra \cite{MinacPasiniQuadrelliTan,Voevodsky} and satisfies strong homological regularity properties, such as Koszulity or universal Koszulity; see, for example, \cite{MinacPalaistiPasiniTan,PositselskiKoszul,Quadrelli}. 
	
	Such properties are formulated when the degree-$1$ part of the algebra is finite-dimensional, i.e. when the algebra is of finite-type. However, for many profinite groups arising in Galois theory, and in particular, for many infinite extensions, the cohomology algebra is no longer of finite-type. This raises the problem of how to retain meaningful Koszul-type structure beyond the finite-dimensional setting. 
	
	The goal of this paper is to introduce and study a notion that captures the residual finite-type linearity present in such situations. The guiding principle is that, even when a graded algebra is too large to be universally Koszul in the classical finite-type sense, one may still hope to recover universal Koszulity on every finite collection of degree-$1$ classes.

	We formalize this idea by introducing \emph{asymptotic universal Koszulity}: a connected graded algebra generated in degree~$1$ is said to be asymptotically universally Koszul if it arises as a filtered colimit of finite-type universally Koszul quadratic subalgebras in a way that captures every finite-dimensional subspace of its degree-$1$ component. We then establish several structural results. First, asymptotic universal Koszulity is preserved under filtered colimits of finite-type universally Koszul quadratic algebras with injective transition maps. Second, for graded-commutative quadratic algebras, we obtain an intrinsic local criterion, formulated in terms of finite-dimensional subspaces of the degree-$1$ component, that implies asymptotic universal Koszulity. Third, we prove a finite-type subalgebra result showing that, under a natural compatibility condition on generators and quadratic relations, certain finitely generated subalgebras inherit universal Koszulity from the approximating finite-type cohomology rings.
	
	We then turn to continuous mod-$p$ cohomology of profinite groups. Under a mild colimit hypothesis--namely, generation in degree-$1$ together with injectivity of inflation in low degrees along a cofinal system of quotients--we show that asymptotic universal Koszulity of the finite-type pieces passes to the full cohomology algebra. This gives a general criterion for showing that \(H^\bullet(G,\mathbb F_p)\) is asymptotically universally Koszul.

	Finally, given a continuous epimorphism \(G \twoheadrightarrow Q\) onto a finitely generated pro-$p$ group, we analyze the image of \(H^\bullet(Q, \mathbb{F}_p)\) inside \(H^\bullet(G, \mathbb{F}_p)\) under inflation. Under injectivity of inflation, we identify this image with the canonical quadratic subalgebra generated by the image of \(H^1(Q, \mathbb{F}_p)\). We also formulate a conditional local-global criterion that isolates the type of finite-dimensional detection needed to deduce asymptotic universal Koszulity in arithmetic settings.

	The results of this paper develop a natural extension of universal Koszulity to infinite-type Galois cohomology. Instead of relying on global homological constraints, the asymptotic approach extracts finite-dimensional structural data and assembles it into a filtered framework. We hope that our perspective opens the door to new connections between homological algebra, pro-$p$ Galois theory, and local-global principles.	
	
	We conclude with several questions concerning the scope of asymptotic universal Koszulity, its possible strengthening to obstruction results for Galois groups, and its behavior in arithmetic settings arising from patching and from composita constructions such as those studied in \cite{PalaistiComposita}.

	\section{Quadratic algebras and universal Koszulity}

	We briefly recall the algebraic notions that underlie the asymptotic framework developed in this paper. Many mod-$p$ cohomology rings arising in Galois theory are naturally graded-commutative quadratic algebras, and several important finiteness and linearity properties are most naturally formulated in that setting. Among these, universal Koszulity provides a strong finite-type regularity condition that will serve as the basic local model throughout the paper. For background on quadratic and Koszul algebras and the relevant homological framework, see \cite{Froberg,PolishchukPositselski,Weibel}.
	
	\begin{definition}
		A connected graded \(k\)-algebra \(A = \displaystyle \bigoplus_{n\geq 0} A_n\), generated in degree \(1\), is \emph{quadratic} if \[ A \cong T(V)/\langle R \rangle, \] for a (possibly infinite-dimensional) vector space \(V = A_1\) and a subspace \(R \subset V \otimes_k V\).
	\end{definition}
	
	Universal Koszulity is the finite-type regularity condition that will be approximated asymptotically in the sequel.
	
	\begin{definition}\label{def:UK}
		Let \(A\) be a quadratic algebra with \(\dim_k A_1 < \infty\). Let
		\[ \mathcal L(A) := \{\, I\subseteq A \mid I \text{ is a graded left ideal and } I = A I_1\,\}\] be the set of graded left ideals generated in degree \(1\). We say that \(A\) is \emph{universally Koszul} if for every \(I\in\mathcal L(A)\) and every \(x \in A_1 \setminus I_1\), the left colon ideal \[ I:x \;:=\; \{\,a \in A \mid ax \in I\,\}\] 	belongs to \(\mathcal L(A)\) (equivalently, \(I:x\) is generated in degree \(1\)).
	\end{definition}

	\begin{remark}
		If \(A\) is graded-commutative (as in mod-\(p\) cohomology rings), then ``graded left ideal''	coincides with ``graded ideal'', and the colon ideal \(I:x\) is two-sided.
	\end{remark}

	Restricting attention to quadratic algebras whose degree-$1$ part is finite-dimensional, universal Koszulity admits several equivalent formulations, including a characterization in terms of colon ideals and linear resolutions of cyclic modules (see, for instance, \cite{PalaistiThesis}). In particular, every cyclic module \(A/I\), with \(I\) generated in degree \(1\), admits a linear minimal free resolution, a perspective that is often useful in establishing stability properties in concrete situations. These equivalences and homological consequences depend fundamentally on finiteness assumptions and do not extend to infinite-type settings, which necessitates the asymptotic framework developed in the next section.

	\section{Asymptotic universal Koszulity}
	
	We now introduce the central notion of this paper. The guiding idea is that an infinite-type graded algebra may fail to satisfy universal Koszulity globally, while still being well approximated by finite-type universally Koszul pieces.
	
	Throughout this section, let \(k\) be a field. The following definition formalizes this principle by requiring a filtered system of finite-type universally Koszul quadratic subalgebras that captures every finite-dimensional part of the degree-$1$ structure. 
	
	\begin{definition}\label{def:asymp}
		Let \(A\) be a connected graded $k$-algebra generated in degree \(1\). We say that \(A\) is \emph{asymptotically universally Koszul} if there exists a filtered system of quadratic graded subalgebras \( \displaystyle (A_i)_{i\in I}\subset A\), such that:
		\begin{enumerate}[label=(\roman*)]
			\item each \(A_i\) is generated in degree \(1\) with \(\dim_k (A_i)_1 < \infty\);
			\item each \(A_i\) is universally Koszul;
			\item the natural map \(\displaystyle \varinjlim_{i\in I} A_i \to A\) is an isomorphism of graded algebras;
			\item every finite-dimensional subspace of \(A_1\) is contained in some \((A_i)_1\).
		\end{enumerate}
	\end{definition}
	
	Intuitively, the asymptotic condition asks that every finite collection of degree-$1$ classes should already live inside a finite-type universally Koszul piece, and that these pieces fit together coherently across enlargements.
	
	The first basic observation is that the defining approximation property is stable under filtered colimits of finite-type universally Koszul algebras with injective transition maps. In particular, the new notion applies naturally to algebras built from compatible finite-dimensional pieces.
	
	\begin{proposition}\label{prop:colim}
		Let \((A_i)_{i\in I}\) be a filtered system of quadratic graded $k$-algebras with injective transition maps, such that each \(A_i\) is of finite-type and universally Koszul. Then \(A := \displaystyle \varinjlim_{i\in I} A_i\) is asymptotically universally Koszul.
	\end{proposition}
	
	\begin{proof}
		We identify each \(A_i\) with its image in \(A\) via the injective maps. Let \(W\subseteq A_1\) be finite-dimensional and choose a basis \(\{w_1,\dots,w_r\}\) of \(W\). For each \(j\), pick \(i_j\in I\), such that \(w_j\in (A_{i_j})_1\). Since \(I\) is filtered, there exists \(k\in I\), such that \(i_j\le k\) for all \(j\), hence \((A_{i_j})_1\subseteq (A_k)_1\) for all \(j\). Therefore \(W\subseteq (A_k)_1\). By construction of the filtered colimit in the category of graded algebras, the canonical map \(\varinjlim_i A_i\to A\) is an isomorphism. Thus all conditions in Definition~\ref{def:asymp} are satisfied.
	\end{proof}
	
	We record the following immediate consequence for exterior algebras.
	
	\begin{corollary}
		Let \(k\) be a field, and let \(V\) be an arbitrary \(k\)-vector space. Then the exterior algebra
		$\bigwedge(V)$ is asymptotically universally Koszul.
	\end{corollary}
	
	\begin{proof}
		Let \(I\) be the directed poset of finite-dimensional subspaces \(W\subseteq V\), ordered by inclusion. Then
		\[V=\varinjlim_{W\in I} W, \qquad \bigwedge(V)=\varinjlim_{W\in I}\bigwedge(W),\]
		where the transition maps are the natural injective graded algebra homomorphisms.
		
		For each finite-dimensional subspace \(W\subseteq V\), the exterior algebra \(\bigwedge(W)\) is a finite-type quadratic algebra, and it is universally Koszul. Thus Proposition~\ref{prop:colim} applies to the filtered system
		\[\bigl(\bigwedge(W)\bigr)_{W\in I},\]
		and yields that \(\bigwedge(V)\)is asymptotically universally Koszul.
	\end{proof}

	The preceding corollary provides a canonical family of examples. We now give an intrinsic criterion for asymptotic universal Koszulity in terms of finite-dimensional subalgebras. Proposition~\ref{prop:colim} applies when a filtered colimit presentation is already given. The next result gives a more intrinsic criterion, showing that it is enough to produce sufficiently many finite-type universally Koszul subalgebras containing prescribed finite-dimensional degree-$1$ data.

	\begin{proposition}\label{prop:local}
		Let \(A\) be a connected graded $k$-algebra generated in degree \(1\). Assume that for every finite-dimensional subspace \(W \subset A_1\), there exists a quadratic subalgebra \(B_W \subset A\), such that:
		\begin{enumerate}[label=(\alph*)]
			\item \(W \subset (B_W)_1\) and \(\dim_k (B_W)_1 < \infty\);
			\item \(B_W\) is universally Koszul;
		\end{enumerate}
		and assume moreover that the family \((B_W)\) can be chosen to be filtered under inclusion. Then \(A\) is asymptotically universally Koszul.
	\end{proposition}
	
	\begin{proof}
		By assumption, every finite-dimensional subspace of \(A_1\) is contained in some \((B_W)_1\). Since \(A\) is generated in degree \(1\), every homogeneous element of \(A\) lies in a subalgebra generated by finitely many degree-$1$ elements, hence lies in some \(B_W\). Therefore \[ A =\bigcup_W B_W \] as graded algebras. Since the family \((B_W)\) is filtered under inclusion,	the natural map \(\displaystyle \varinjlim_W B_W \longrightarrow A\) is an isomorphism. The conclusion follows from Definition~\ref{def:asymp}.
	\end{proof}
	
	Please note that the filteredness assumption in Proposition~\ref{prop:local} is essential. In general, the existence of finite-type universally Koszul subalgebras containing prescribed finite-dimensional subspaces does not guarantee that these subalgebras can be arranged into a filtered system under inclusion.
	
	This naturally leads to the question of when the local finite-type pieces in Proposition~\ref{prop:local} can be taken canonically.

	\section{Local finite-type structure}
	\label{sec:local}
	
	We now give an intrinsic formulation of asymptotic universal Koszulity in terms of finite-dimensional subspaces of the degree-$1$ part. Throughout this section, let $k$ be a field, and let $A = \bigoplus_{n \ge 0} A_n$ be a connected graded-commutative quadratic $k$-algebra.

	For any finite-dimensional subspace $W \subseteq A_1$, we denote by
	\[ B_W := \langle W \rangle_A\]
	the graded subalgebra generated by $W$. Since $A$ is quadratic, $B_W$
	is again quadratic, with relation space
	\[R_{B_W} = R_A \cap (W \otimes W).\]
	The point is that, in the graded-commutative quadratic setting, these canonical subalgebras provide a distinguished local approximation to \(A\) attached to each finite-dimensional part of \(A_1\).

	The next proposition shows that if these canonical local approximants are universally Koszul, then they already generate the desired finite-type asymptotic structure.
	
	\begin{proposition}
		\label{prop:b-implies-c}
		Let $A$ be a connected graded-commutative quadratic $k$-algebra. Assume that for every finite-dimensional subspace $W \subseteq A_1$, the canonical quadratic subalgebra $B_W$ is universally Koszul. Then $A$ is the directed union of finite-type universally Koszul quadratic subalgebras.
	\end{proposition}
	
	\begin{proof}
		Let $\mathcal{W}$ be the poset of finite-dimensional subspaces of $A_1$, ordered by inclusion. This poset is directed; indeed, given $W, W'$, the sum $W + W'$ is again finite-dimensional and contains both. For each $W \in \mathcal{W}$, the algebra $B_W$ is quadratic, and finite-type by construction. By assumption, it is universally Koszul. Moreover, if $W \subseteq W'$, then $B_W \subseteq B_{W'}$, since $B_W = \langle W \rangle_A \subseteq \langle W' \rangle_A = B_{W'}$. Thus the family $(B_W)_{W \in \mathcal{W}}$ forms a directed system. Finally, since $A$ is generated in degree-$1$, every homogeneous element of $A$ lies in the subalgebra generated by finitely many elements of $A_1$, hence in some $B_W$. Therefore
		\[	A = \bigcup_{W \in \mathcal{W}} B_W,\]
		and the union is directed. The result follows.
	\end{proof}

	We now observe that any directed union of finite-type universally Koszul quadratic subalgebras automatically satisfies the asymptotic property.
	
	\begin{proposition}\label{prop:c-implies-a}
		Let $A$ be a connected graded algebra, and assume that $A$ is the directed union of finite-type universally Koszul quadratic subalgebras. Then $A$ is asymptotically universally Koszul.
	\end{proposition}
	
	\begin{proof}
		This is immediate from Definition~\ref{def:asymp}. The given directed system satisfies conditions \textrm{(i)}-\textrm{(iv)} by construction.
	\end{proof}

	Propositions~\ref{prop:b-implies-c} and~\ref{prop:c-implies-a} show that a local condition on finite-dimensional subspaces of $A_1$ implies asymptotic universal Koszulity. In particular, if all canonical quadratic subalgebras $B_W$ are universally Koszul, then $A$ is asymptotically universally Koszul. The converse direction, namely whether asymptotic universal Koszulity forces each $B_W$ to be universally Koszul, appears to require a descent principle for quadratic subalgebras and is not established here.

	\section{Compatible finite-type subalgebras}
	\label{sec:structural}
	
	The aim of this section is to establish a structural result about asymptotically universally Koszul algebras under a compatibility condition. Throughout, let $k$ be a field, and all graded algebras are assumed to be connected, generated in degree-$1$, and graded-commutative. 
	
	For the asymptotic framework to be useful, one would like finite-type universally Koszul pieces to pass to distinguished finitely generated subalgebras whenever the ambient approximation is sufficiently compatible.
	
	In general, universal Koszulity does not descend to arbitrary quadratic subalgebras. The following result isolates a natural situation where it does.
	
	\begin{lemma}
		\label{lem:quadratic-subalgebra-compatible}
		Let $C$ be a finite-type universally Koszul quadratic $k$-algebra. Assume that \[C_1 = W \oplus U\]
		as $k$-vector spaces, and that the quadratic relation space of \(C\) decomposes as
		\[R_C = R_W \oplus R_U,\quad \text{ with } \quad R_W \subseteq W \otimes W,\; R_U \subseteq U \otimes U.	\]
		Let $D := \langle W \rangle \subseteq C$ be the graded subalgebra generated by $W$. Then $D$ is a finite-type universally Koszul quadratic algebra.
	\end{lemma}
	
	\begin{proof}
		Under the splitting assumption, there are no quadratic relations mixing $W$ and $U$, so $D$ identifies with the quadratic algebra 
		\[D\cong T(W)/\langle R_W\rangle.\]
		Let $J \in \mathcal{L}(D)$ and $y \in W \setminus J_1$. Let $I:= C\cdot J_1 \in \mathcal{L}(C)$. Since $C$ is universally Koszul, the colon ideal $I :_C y$ is generated in degree-$1$. Write $I:_C y = C\cdot M$, for some subspace $M\subseteq C_1$. Because the presentation of $D$ involves only generators in $W$ and relations in $W\otimes W$, one has 
		\[ J :_D y = (I :_C y) \cap D = D \cdot (M \cap W).\]
		Hence $J:_D y$ is generated in degree-$1$, implying that $D$ is universally Koszul.
	\end{proof}

	We now obtain a subalgebra theorem under a natural compatibility assumption.
	
	\begin{theorem}[Finite-type compatible subalgebras]
		\label{thm:finitely-generated-subalgebra}
		Let $A$ be an asymptotically universally Koszul algebra with filtered system $(A_i)_{i \in I}$ of finite-type universally Koszul quadratic subalgebras. Let $B \subseteq A$ be a connected graded subalgebra generated in degree-$1$ with $\dim_k B_1 < \infty$. Assume that there exists $i \in I$, such that:
		\begin{enumerate}[label=(\roman*)]
			\item $B_1 \subseteq (A_i)_1$,
			\item $(A_i)_1 = B_1 \oplus U$ for some subspace $U$,
			\item the quadratic relation space of $A_i$ decomposes as
			\[ R_{A_i} = R_B \oplus R_U, \quad \text{ with } \quad R_B\subseteq B_1\otimes B_1,\; R_U\subseteq U\otimes U. \] 
		\end{enumerate}
		Then $B$ is a finite-type universally Koszul quadratic algebra.
	\end{theorem}
	
	\begin{proof}
		By assumption, $B$ is a quadratic subalgebra of $A_i$ satisfying the hypotheses of Lemma~\ref{lem:quadratic-subalgebra-compatible}. Since $A_i$ is finite-type universally Koszul, the lemma implies that $B$ is universally Koszul. Finite generation follows from $\dim_k B_1 < \infty$.
	\end{proof}
	
	\begin{remark}
		The assumptions of Theorem~\ref{thm:finitely-generated-subalgebra} are automatically satisfied when $B_1$ is a coordinate subspace of $(A_i)_1$, i.e.\ when $(A_i)_1$ admits a basis, such that $B_1$ is spanned by a subset and the relations respect this decomposition. This occurs in many natural constructions, including certain filtered systems arising in Galois cohomology.
	\end{remark}

	\section{Galois cohomology}
	\label{sec:galois}

	We now apply the preceding framework to continuous mod-\(p\) cohomology of pro-\(p\) groups and, in particular, to maximal pro-\(p\) Galois groups of fields. Since quadraticity and universal Koszul phenomena are governed by the degree-\(1\) and degree-\(2\) structure of the cohomology ring, low-degree compatibility along suitable quotient systems will be sufficient for our purposes.
	
	Throughout this section, let \(p\) be a prime. All cohomology groups below are continuous cohomology with coefficients in \(\mathbb{F}_p\). For a pro-\(p\) group \(G\), we write 
	\[H^\bullet(G):=H^\bullet(G,\mathbb{F}_p),\]
	viewed as a graded-commutative \(\mathbb{F}_p\)-algebra via the cup product. For a field \(K\), let \(G_K(p)\) denote the maximal pro-\(p\) quotient of the absolute Galois group of \(K\). For general background on group and Galois cohomology, pro-\(p\) Galois groups, and related Galois theory, see \cite{Brown,Koch,NeukirchSchmidtWingberg,SerreGaloisCohomology,SerreTopics}.
	
	\begin{definition}
		We say that a field \(K\) has \emph{asymptotically universally Koszul cohomology at \(p\)} if the graded algebra $H^\bullet(G_K(p),\mathbb{F}_p)$ is asymptotically universally Koszul.
	\end{definition}
	
	The first step is a low-degree colimit statement showing that, under suitable injectivity assumptions, the cohomology of \(G\) is recovered from its quotients in the range relevant to quadratic structure.
	
	\begin{lemma}\label{lem:cohom-colimit}
		Let \(G\) be a pro-\(p\) group and let \((N_i)_{i\in I}\) be a cofinal directed system of closed normal subgroups of \(G\), such that
		\[\bigcap_{i\in I} N_i = 1.\]
		Set \(G_i:=G/N_i\). Assume:
		\begin{enumerate}[label=(\roman*)]
			\item \(H^\bullet(G)\) is generated as an \(\mathbb{F}_p\)-algebra by \(H^1(G)\),
			\item for each \(i\in I\) and each \(n\in\{1,2\}\), the inflation map 
			\[\mathrm{inf}^n_{G_i}\colon H^n(G_i)\longrightarrow H^n(G)\]
			is injective.
		\end{enumerate}
		Then the inflation maps induce:
		\begin{enumerate}[label=(\alph*)]
			\item isomorphisms
			\[\varinjlim_{i\in I} H^n(G_i)\xrightarrow{\sim} H^n(G) \qquad \text{for } n=0,1,2;\]
			\item surjections
			\[ \varinjlim_{i\in I} H^n(G_i)\twoheadrightarrow H^n(G) \qquad \text{for all } n\ge 0,\]
			compatible with cup products.
		\end{enumerate}
		In particular, the induced homomorphism of graded \(\mathbb{F}_p\)-algebras
		\[\varinjlim_{i\in I} H^\bullet(G_i)\longrightarrow H^\bullet(G)\]
		is surjective and is an isomorphism in degrees \(\le 2\).
	\end{lemma}
	
	\begin{proof}
		In degree \(0\), one has \(H^0(G_i)\cong \mathbb F_p\cong H^0(G)\) for all \(i\), so the colimit map is an isomorphism in degree \(0\).
		
		In degree \(1\), since \(G=\varprojlim_i G_i\), every continuous homomorphism \(G\to \mathbb{F}_p\) factors through some quotient \(G_i\). Hence 
		\[H^1(G)=\mathrm{Hom}_{\mathrm{cts}}(G,\mathbb{F}_p) =\varinjlim_i \mathrm{Hom}(G_i,\mathbb{F}_p) =\varinjlim_i H^1(G_i),\]
		see \cite[Prop.~1.5.1]{NeukirchSchmidtWingberg}.
		
		Now let \(n\ge 0\). Since \(H^\bullet(G)\) is generated as an algebra by \(H^1(G)\), every element of \(H^n(G)\) is a finite sum of cup-products of degree-$1$ classes. Each degree-$1$ factor lifts from some \(H^1(G_i)\), and because the system is directed, all factors occurring in a given finite sum lift simultaneously from some common \(G_{i_0}\). Compatibility of inflation with cup products then shows that the class lies in the image of \(H^n(G_{i_0})\to H^n(G)\). Thus for every \(n\ge 0\), the map 
		\[ \varinjlim_i H^n(G_i)\longrightarrow H^n(G)\]
		is surjective.
		
		In degrees \(1\) and \(2\), the maps \[ \mathrm{inf}^n_{G_i}\colon H^n(G_i)\to H^n(G) \]
		are injective by assumption. Hence the canonical maps from the filtered colimits are injective in degrees \(1\) and \(2\). Combined with the surjectivity, this yields isomorphisms
		\[\varinjlim_i H^n(G_i)\xrightarrow{\sim} H^n(G) \qquad (n=1,2).\]
		
		Finally, inflation is functorial and compatible with cup products, so the graded colimit map is a surjective homomorphism of graded algebras and an isomorphism in degrees \(\le 2\).
	\end{proof}
	
	We now use the low-degree control provided by Lemma~\ref{lem:cohom-colimit} together with quadraticity of the finite quotients, to obtain the following.
	
	\begin{proposition}\label{prop:galois}
		Let \(G\) be a pro-\(p\) group and let \((N_i)_{i\in I}\) be a cofinal directed system of closed normal subgroups of \(G\) with \(\bigcap_i N_i=1\). Set \(G_i:=G/N_i\). Assume:
		\begin{enumerate}[label=(\roman*)]
			\item \(H^\bullet(G)\) is generated as an \(\mathbb{F}_p\)-algebra by \(H^1(G)\),
			\item for each \(i\) and each \(n\in\{1,2\}\), the inflation map
			\[\mathrm{inf}^n_{G_i}\colon H^n(G_i)\longrightarrow H^n(G)\]
			is injective,
			\item for each \(i\), the algebra \(H^\bullet(G_i)\) is a finite-type universally Koszul \(\mathbb{F}_p\)-algebra.
		\end{enumerate}
		Then \(H^\bullet(G)\) is asymptotically universally Koszul.
	\end{proposition}
	
	\begin{proof}
		For each \(i\), let
		\[W_i:=\mathrm{inf}^1_{G_i}\bigl(H^1(G_i)\bigr)\subseteq H^1(G),\] and let
		\[B_i:=\langle W_i\rangle_{H^\bullet(G)}\]
		denote the graded subalgebra of \(H^\bullet(G)\) generated by \(W_i\). We claim that the inflation map induces an isomorphism of quadratic algebras
		\[H^\bullet(G_i)\xrightarrow{\sim} B_i.\]
		Indeed, since \(H^\bullet(G_i)\) is quadratic, it is generated by \(H^1(G_i)\), so its image is precisely \(B_i\). It remains to identify the quadratic relation space. In \(H^\bullet(G_i)\), the quadratic relations are
		\[R_i=\ker\!\Bigl(H^1(G_i)^{\otimes 2}\xrightarrow{\cup} H^2(G_i)\Bigr).\]
		Since inflation is injective on \(H^1\) and \(H^2\), and is compatible with cup products, the image of \(R_i\) inside \(W_i^{\otimes 2}\) is exactly
		\[\ker\!\Bigl(W_i^{\otimes 2}\xrightarrow{\cup} H^2(G)\Bigr),\]
		which is the quadratic relation space of \(B_i\). Hence $H^\bullet(G_i)\cong B_i$ as quadratic algebras. Therefore each \(B_i\) is a finite-type universally Koszul quadratic subalgebra of \(H^\bullet(G)\).
		
		Next, if \(N_j\subseteq N_i\), then the quotient map \(G_j\twoheadrightarrow G_i\) implies $W_i\subseteq W_j$, hence $B_i\subseteq B_j$. hus the family \((B_i)\) is filtered under inclusion.
		
		Finally, by Lemma~\ref{lem:cohom-colimit}, every class in \(H^1(G)\) comes from some \(H^1(G_i)\), so \[H^1(G)=\bigcup_i W_i.\]
		Since \(H^\bullet(G)\) is generated by \(H^1(G)\), every homogeneous element of \(H^\bullet(G)\) lies in some \(B_i\). Therefore \[H^\bullet(G)=\bigcup_i B_i.\] Hence \(H^\bullet(G)\) is a filtered union of finite-type universally Koszul quadratic subalgebras. By Proposition~\ref{prop:local}, it follows that \(H^\bullet(G)\) is asymptotically universally Koszul.
	\end{proof}

	\begin{remark}
		In the setting of Proposition~\ref{prop:galois}, injectivity of inflation in degree \(1\) is automatic. In degree \(2\), sufficient conditions can often be obtained from the five-term exact sequence associated to the Lyndon--Hochschild--Serre spectral sequence for $1\to N_i\to G\to G_i\to 1$. For example, it suffices that $H^1(N_i,\mathbb F_p)^{G_i}=0$. 
	\end{remark}

	We now apply Proposition~\ref{prop:galois} to a basic non-finite-type class of pro-$p$ groups, namely free pro-$p$ groups of arbitrary rank.
	
	\begin{proposition}
		Let \(F\) be a free pro-\(p\) group on an arbitrary set \(X\) of free generators. Then $H^\bullet(F,\mathbb{F}_p)$ is asymptotically universally Koszul.
	\end{proposition}
	
	\begin{proof}
		Let \(I\) be the directed poset of finite subsets \(S\subseteq X\), ordered by inclusion. For each
		\(S\in I\), let \(F_S\) denote the free pro-\(p\) group on the finite set \(S\), and let
		\[\rho_S\colon F\twoheadrightarrow F_S\] be the continuous epimorphism determined by
		\[\rho_S(x)=\begin{cases} x,& x\in S,\\ 1,& x\in X\setminus S. \end{cases} \]
		Set
		\[N_S:=\ker(\rho_S).\]
		Then each \(N_S\) is a closed normal subgroup of \(F\), and
		\[F/N_S \cong F_S.\]
		
		We verify the hypotheses of Proposition~\ref{prop:galois}. First, the family \((N_S)_{S\in I}\) is directed under reverse inclusion: if \(S\subseteq T\), then \(\rho_S\) factors through \(\rho_T\), hence $N_T\subseteq N_S.$ Moreover, \[ \bigcap_{S\in I} N_S = 1. \]
		Indeed, let \(g\in F\) be nontrivial. Since \(F\) is free pro-\(p\) on \(X\), there exists a continuous homomorphism from \(F\) to a finite \(p\)-group \(P\) whose value on \(g\) is nontrivial. As \(P\) is finite, only finitely many generators in \(X\) have nontrivial image. Let \(S\subseteq X\) be the finite set of such generators. Then this homomorphism factors through \(\rho_S\), so \(\rho_S(g)\neq 1\). Hence \(g\notin N_S\), proving that \(\bigcap_{S\in I}N_S=1\).
		
		Next, since \(F\) is a free pro-\(p\) group, it has cohomological dimension \(1\), thus $	H^n(F,\mathbb F_p)=0 \qquad (n\geq 2)$, so \(H^\bullet(F,\mathbb F_p)\) is generated as an \(\mathbb{F}_p\)-algebra by \(H^1(F,\mathbb{F}_p)\).
		
		For each finite subset $S \subseteq X$, the quotient $F_S$ is a
		finitely-generated free pro-$p$ group, hence also has cohomological
		dimension $1$. Thus
		\[
		H^n(F_S,\mathbb{F}_p)=0 \qquad (n\geq 2).
		\]
		Writing $V_S:=H^1(F_S,\mathbb{F}_p)$, it follows that
		\[
		H^\bullet(F_S,\mathbb{F}_p)
		\cong T(V_S)/\langle V_S\otimes V_S\rangle,
		\]
		the trivial quadratic algebra on the finite-dimensional vector space
		$V_S$. In particular, $H^\bullet(F_S,\mathbb{F}_p)$ is a finite-type
		universally Koszul $\mathbb{F}_p$-algebra; see \cite{Quadrelli}.
		
		Finally, inflation is automatically injective in degrees \(1\) and \(2\). Thus all hypotheses of Proposition~\ref{prop:galois} are satisfied for the cofinal directed system \((N_S)_{S\in I}\). Hence $H^\bullet(F,\mathbb{F}_p)$ is asymptotically universally Koszul.
	\end{proof}

	Applying the preceding Proposition with \(F=G_K(p)\), we obtain
	\begin{corollary}[Fields with free maximal pro-\(p\) Galois group]
		Let \(K\supset \mu_p\) be a field, and assume that \(G_K(p)\) is a free pro-\(p\) group of arbitrary rank. Then \(K\) has asymptotically universally Koszul cohomology at \(p\).
	\end{corollary}

We next give an example arising from compact $\mathbb Z_p$-Lie algebras.

\begin{example}[Pro-uniform pro-$p$ groups]
	Let $p$ be an odd prime, and let
	\[
	L=\varprojlim_i L_i
	\]
	be a compact $\mathbb Z_p$-Lie algebra, where each $L_i$ has finite $\mathbb Z_p$-rank and the transition maps $L_i\to L_j$ are surjective. Let $\Phi_h$ denote the group law induced by the Baker--Campbell--Hausdorff formula, and set
	\[
	G=(pL,\Phi_h).
	\]
	Then $G$ is a pro-uniform pro-$p$ group, obtained as an inverse limit of uniform pro-$p$ groups associated to the finite-rank $\mathbb Z_p$-Lie algebras $pL_i$. For background on analytic and uniform pro-$p$ groups and their Lie-theoretic description, see \cite{DixonDuSautoyMannSegal}. Its mod-$p$ cohomology algebra is
	\[
	H^\bullet(G,\mathbb F_p)\cong \bigwedge(V),
	\]
	where
	\[
	V=(G/G^p)^\vee \cong (L/pL)^\vee.
	\]
	It follows from Corollary~3.3 that $H^\bullet(G,\mathbb F_p)$ is asymptotically universally Koszul.
\end{example}

	A natural question is whether the finite-type pieces appearing in an asymptotically universally Koszul cohomology algebra can be chosen in a group-theoretically meaningful way.
	
	The next result shows that, whenever inflation is injective, the cohomology of such a quotient appears canonically as a finite-type quadratic subalgebra generated by its degree-\(1\) image. 
	
	\begin{theorem}[Finite-type capture for quotient images]\label{thm:quotient-obstruction}
		Let \(G\) be a pro-\(p\) group, such that \(H^\bullet(G)\) is asymptotically universally Koszul via a filtered system \((A_i)_{i\in I}\) of finite-type universally Koszul quadratic subalgebras with
		\[\varinjlim_{i\in I} A_i \cong H^\bullet(G).\]
		Let \(\pi\colon G\twoheadrightarrow Q\) be a continuous epimorphism, where \(Q\) is a finitely generated pro-\(p\) group, such that \(H^\bullet(Q)\) is quadratic. Let
		\[\pi^*\colon H^\bullet(Q)\longrightarrow H^\bullet(G)\]
		denote the inflation, and assume that \(\pi^*\) is injective. Set
		\[W:=\pi^*(H^1(Q))\subseteq H^1(G),\]
		and let
		\[B_W:=\langle W\rangle_{H^\bullet(G)}\]
		be the graded subalgebra generated by \(W\).
		Then:
		\begin{enumerate}
			\item there exists \(i\in I\), such that \(W\subseteq (A_i)_1\),
			\item one has \(B_W\subseteq A_i\),
			\item the image \(\pi^*(H^\bullet(Q))\) is equal to \(B_W\);
			\item under the identification \(H^1(Q)\cong W\) induced by \(\pi^*\), the quadratic relation space of \(B_W\) is the image of the quadratic relation space of \(H^\bullet(Q)\). In particular,
			\[\pi^*\colon H^\bullet(Q)\xrightarrow{\sim} B_W\] is an isomorphism of quadratic algebras.
		\end{enumerate}
	\end{theorem}
	
	\begin{proof}
		Since \(Q\) is finitely generated, the vector space \(H^1(Q)\) is finite-dimensional. Hence \(W\subseteq H^1(G)\) is finite-dimensional. Because \(H^\bullet(G)\) is asymptotically universally Koszul, there exists \(i\in I\), such that $W\subseteq (A_i)_1$. This proves (1).
		
		Since $H^\bullet(Q)$ is quadratic, it is generated as an algebra by $H^1(Q)$, thus 
		\[\pi^*(H^\bullet(Q)) = \left\langle \pi^* (H^1(Q))\right\rangle_{H^\bullet(G)}=\langle W \rangle_{H^\bullet(G)}=B_W.\]
		Because $W\subseteq (A_i)_1$, one has $B_W\subseteq A_i$, thus proving (2) and (3). 
		
		It remains to identify the quadratic relations. Since $H^\bullet(Q)$ is quadratic, its relation space is 
		\[R_Q= \ker\!\Bigl(H^1(Q)^{\otimes 2}\xrightarrow{\cup} H^2(Q)\Bigr).\]
		Since \(\pi^*\) is injective on \(H^1(Q)\) and \(H^2(Q)\) and compatible with cup products, the image of \(R_Q\) inside \(W^{\otimes 2}\) is exactly
		\[\ker\!\Bigl(W^{\otimes 2}\xrightarrow{\cup} H^2(G)\Bigr),\]
		which is the quadratic relation space of \(B_W\). Thus $ \pi^*\colon H^\bullet(Q)\xrightarrow{\sim} B_W$ is an isomorphism of quadratic algebras.
	\end{proof}

	\begin{remark}
		Theorem~\ref{thm:quotient-obstruction} shows that whenever  $G \twoheadrightarrow Q$ and inflation is injective, $H^\bullet(Q)$ shows up canonically inside $H^\bullet(G)$ as the quadratic subalgebra generated by the image of $H^1(Q)$. Thus finitely-generated quotients of $G$ provide finite-type quadratic pieces of the ambient cohomology algebra. This perspective may be useful when one attempts to establish asymptotic universal Koszulity via suitable families of quotients.
	\end{remark}

	\begin{corollary}
		Let \(G\) be a pro-\(p\) group such that \(H^\bullet(G)\) is asymptotically universally Koszul. Let
		\[\{\pi_\lambda \colon G \twoheadrightarrow Q_\lambda\}_{\lambda\in\Lambda}\]
		be a family of continuous epimorphisms onto finitely generated pro-\(p\) groups such that, for each \(\lambda\in\Lambda\),
		\begin{enumerate}
			\item[\textrm{(i)}] \(H^\bullet(Q_\lambda)\) is quadratic, and
			\item[\textrm{(ii)}] the inflation map
			\[\pi_\lambda^* \colon H^\bullet(Q_\lambda)\longrightarrow H^\bullet(G)\]
			is injective.
		\end{enumerate}
		Then, for every \(\lambda\in\Lambda\), the image \(\pi_\lambda^*(H^\bullet(Q_\lambda))\) is a finite-type quadratic subalgebra of \(H^\bullet(G)\), canonically generated by the image of \(H^1(Q_\lambda)\).
	\end{corollary}
	
	\begin{proof}
		Apply Theorem~\ref{thm:quotient-obstruction} to each quotient map $\pi_\lambda \colon G\twoheadrightarrow Q_\lambda$. For each \(\lambda\), the theorem identifies $	\pi_\lambda^*(H^\bullet(Q_\lambda))$ with the quadratic subalgebra of \(H^\bullet(G)\) generated by the image of \(H^1(Q_\lambda)\). Since \(Q_\lambda\) is finitely-generated, this subalgebra is finite-type.
	\end{proof}

	We conclude this section with a basic finite-type arithmetic example illustrating the terminology introduced above.
	
	\begin{example}[Finite-type case: local fields]
		Let \(p\) be a prime and let \(K\) be a nonarchimedean local field containing \(\mu_p\), and set $	G:=G_K(p)$. Then \(G\) is a finitely generated pro-\(p\) group of cohomological dimension \(2\), and \(H^\bullet(G,\mathbb{F}_p)\) is a finite-type quadratic algebra. In many such cases it is known to be universally Koszul. Hence, whenever \(H^\bullet(G_K(p),\mathbb{F}_p)\) is universally Koszul, it is automatically asymptotically universally Koszul. Thus every local field \(K\supseteq \mu_p\) whose maximal pro-\(p\) Galois cohomology is universally Koszul gives a basic example of a field with asymptotically universally Koszul cohomology at \(p\).
	\end{example}

	\section{Conditional local-global descent}
	
	We now formulate a conditional local-global mechanism for establishing asymptotic universal Koszulity. The key idea is that if finite-dimensional collections of global degree-$1$ classes can be realized inside finite-type universally Koszul local models in a compatible way, then the global cohomology ring inherits asymptotic universal Koszulity.

	\begin{theorem}[Conditional local-global descent criterion] \label{thm:local-global-descent}
		Let $K \supset \mu_p$ be a field, and set
		\[A_K^\bullet := H^\bullet(G_K(p),\mathbb{F}_p).\]
		Assume that for every finite-dimensional subspace $W \subset A_K^1$, there exist:
		\begin{enumerate}[label=(\roman*)]
			\item a finite set of places or local data $S=S(W)$,
			\item a graded algebra homomorphism
			\[\lambda_S \colon A_K^\bullet \longrightarrow A_S^\bullet := \prod_{v\in S} H^\bullet(G_{K_v}(p),\mathbb F_p), \]
			\item a finite-type universally Koszul quadratic algebra $B_W$,
			\item an isomorphism $\phi_W \colon \langle W\rangle \xrightarrow{\ \sim\ } B_W$,
		\end{enumerate}
		such that:
		\begin{enumerate}[label=\textnormal{(\alph*)}]
			\item $\phi_W$ is induced by $\lambda_S$,
			\item the family is compatible under inclusion, that is, 
			\[W \subset W' \Rightarrow \langle W\rangle \subset \langle W'\rangle, \qquad B_W \subset B_{W'}.\]
		\end{enumerate}
		Then $A_K^\bullet$ is a filtered union of finite-type universally Koszul quadratic subalgebras, and hence asymptotically universally Koszul.
	\end{theorem}
	
	\begin{proof}
		For each $W$, let $C_W := \langle W\rangle$. By assumption, $C_W \cong B_W$, hence $C_W$ is finite-type and universally Koszul. Compatibility implies that $(C_W)$ is a filtered system. Since $A_K^\bullet$ is generated in degree $1$, every element lies in some $C_W$, and thus
		\[A_K^\bullet = \bigcup_W C_W.\]
		Hence $A_K^\bullet$ is a filtered colimit of finite-type universally Koszul quadratic algebras. The result follows from Proposition~\ref{prop:local}.
	\end{proof}

	The preceding theorem reduces asymptotic universal Koszulity to a finite-dimensional detection problem: each finite collection of global degree-$1$ classes must be recoverable inside a finite-type universally Koszul local model, and these models must be compatible under enlargement.

	\begin{corollary}
		\label{cor:finite-local-product-descent}
		Let $K \supset \mu_p$ be a field. Assume that for every finite-dimensional subspace $W \subseteq H^1(G_K(p),\mathbb{F}_p)$ there exists a finite set $S$ and a map
		\[\lambda_S \colon H^\bullet(G_K(p),\mathbb{F}_p) \longrightarrow \prod_{v\in S} H^\bullet(G_{K_v}(p),\mathbb{F}_p),\]
		such that:
		\begin{enumerate}[label=(\roman*)]
			\item $\lambda_S$ is injective on $\langle W\rangle$,
			\item $\lambda_S(\langle W\rangle)$ is finite-type universally Koszul,
			\item compatibility holds under inclusion of subspaces.
		\end{enumerate}
		Then $H^\bullet(G_K(p),\mathbb F_p)$ is asymptotically universally Koszul.
	\end{corollary}
	
	We now outline how the preceding result would apply to patching by formulating a criterion that isolates the local-global assumptions required to deduce asymptotic universal Koszulity via patching. For the patching and local-global framework underlying this discussion, see \cite{HarbaterHartmannKrashenLGG,HarbaterHartmannKrashenLGcohom,HarbaterHartmannPatching}. Let $F\supset \mu_p $ be a complete discretely valued field, let $X$ be a smooth projective curve over $F$, and set $K = F(X)$. Assume a patching datum $\Gamma$ is given with local fields $K_v$. Assume there is a natural map
	\[\lambda_\Gamma \colon H^\bullet(G_K(p),\mathbb{F}_p) \longrightarrow
	\prod_{v\in \Gamma} H^\bullet(G_{K_v}(p),\mathbb{F}_p),\]
	whose restriction to any finite-dimensional $\langle W\rangle$ factors through a finite subproduct. Further, assume:
	\begin{enumerate}[label=(\roman*)]
		\item finite-dimensional subspaces embed via $\lambda_S$,
		\item their images are finite-type universally Koszul,
		\item compatibility under inclusion holds.
	\end{enumerate}
	Then $H^\bullet(G_K(p),\mathbb F_p)$ is asymptotically universally Koszul.
	
	\section{Questions}
	
	The preceding results leave open both structural and arithmetic questions concerning the scope of the asymptotic framework.
	
	\begin{enumerate}
		
		\item  Under what additional assumptions does asymptotic universal Koszulity force the canonical quadratic subalgebras \(B_W\) associated to finite-dimensional subspaces \(W\subseteq A_1\) to be universally Koszul? Equivalently, when does the sufficient local criterion of Section~\ref{sec:local} become a genuine characterization?
		
		\item Under what additional compatibility conditions on the filtered system \((A_i)\) does the assumption of Theorem~\ref{thm:finitely-generated-subalgebra} hold automatically for all finitely generated graded subalgebras? 
		
		\item Is \(H^\bullet(G_K(p),\mathbb F_p)\) asymptotically universally Koszul when \(K\supset \mu_p\) is a number field? More generally, does the same hold for an arbitrary field \(K\supset \mu_p\)?
		
		\item Which natural classes of fields \(K\supset \mu_p\) satisfy the assumptions of Proposition~\ref{prop:galois}? 
		
		\item To what extent can the conditional local-global descent principle of Theorem~\ref{thm:local-global-descent} be made unconditional in concrete arithmetic settings, such as function fields of curves over complete discretely valued fields?

	\end{enumerate}

\section*{Acknowledgments} The author would like to thank the anonymous referee for the careful reading of the manuscript and for valuable comments and suggestions which helped improve the content and exposition. The author would also like to thank the Editor, Mehrdad Nasernejad, for the kind encouragement and assistance throughout the publication process.

\end{document}